\documentclass[12pt]{amsart}
\usepackage{amsmath,amsthm,amssymb}
\usepackage{times}
\usepackage{enumerate}
\usepackage{amsfonts, hyperref, graphicx, ifpdf}

\usepackage[dvipsnames,usenames]{color}

\newtheorem{theorem}{Theorem}[section]

\newtheorem{lemma}[theorem]{Lemma}

\newcommand{\be}{\begin{equation}}
\newcommand{\ee}{\end{equation}}
\newcommand{\lt}{\left}
\newcommand{\rt}{\right}
\newcommand{\R}{\mathbb{R}}
\newcommand{\mH}{\mathbb{H}}
\newcommand{\hF}{\hat{F}}
\newcommand{\hG}{\hat{G}}
\newcommand{\tQ}{\tilde{Q}}

\newcommand{\e}{\epsilon}
\newcommand{\lhs}{\text{l.h.s.}}
\newcommand{\rhs}{\text{r.h.s.}}
\newcommand{\bQ}{\bar{Q}}
\newcommand{\tL}{\tilde{\Lambda}}
\newcommand{\La}{\Lambda}
\newcommand{\ga}{\gamma}

\theoremstyle{definition}

\numberwithin{equation}{section}

\begin{document}
\setlength{\baselineskip}{1.2\baselineskip}

\title[Constant curvature graphs in Hyperbolic space]
{Optimal regularity of constant curvature graphs in Hyperbolic space}

\author{Xumin Jiang}
\address{Department of Mathematics, Rutgers University,
Piscataway, NJ 08854}
\email{xj60@math.rutgers.edu}
\author{Ling Xiao}
\address{Department of Mathematics, University of Connecticut,
Storrs, CT 06268}
\email{ling.2.xiao@uconn.edu}

\begin{abstract}
Inspired by \cite{HanJiang, HanJiang2}, we study the boundary regularity of constant curvature hypersurfaces in the hyperbolic space $\mathbb{H}^{n+1},$ which have prescribed asympototic boundary at infinity. Through constructing the boundary expansions of the solutions, we derive the optimal regularity of the solutions. Moreover, we obtain an equivalent condition that guarantees the smoothness of the solutions.
\end{abstract}

\maketitle

\section{Introduction}
\label{sec0}
\bigskip
In \cite{Guan&Spruck&Szapiel, Guan&Spruck, GSX}, Bo Guan, Joel Spruck, and their collaborators studied the existence of a complete strictly convex hypersurface of constant curvature in hyperbolic space with a prescribed asymptotic
boundary at infinity. More specifically, given $\Gamma=\partial\Omega\subset\partial_\infty\mathbb{H}^{n+1}$
and a large class of smooth symmetric functions
$\phi$ of $n$ variables, they proved the existence of a complete strictly convex hypersurface $\Sigma$ in $\mathbb{H}^{n+1}$ satisfying
\begin{equation}
\label{eq0}
\begin{aligned}
\phi(\kappa[\Sigma]) &=\sigma\\
\partial\Sigma = \Gamma
\end{aligned}
\end{equation}
where $\kappa[\Sigma] = (\kappa_1, . . . , \kappa_n)$ denotes the positive hyperbolic principal curvatures of $\Sigma$
and $\sigma\in(0, 1)$ is a constant. They assumed that the hypersurface $\Sigma$ is a vertical graph over the domain
$\Omega,$ i.e.,
\[\Sigma = \{(x, x_{n+1}) : x \in\Omega, x_{n+1} = f(x)\}.\]
Moreover, they proved that $f^2\in C^{\infty}(\Omega)\cap C^{1,1}(\bar{\Omega}).$

In this paper, we will study the optimal regularity of $\Sigma$ on the boundary. We will use the half-space model, i.e.,
\[\mathbb{H}^{n+1}=\{(x, x_{n+1})\in\mathbb{R}^{n+1}| x_{n+1}>0\}\]
equipped with the hyperbolic metric
\[ds_H^2=\frac{1}{x_{n+1}^2}ds_E^2,\]
where $ds_E^2$ denotes the Euclidean metric on $\mathbb{R}^{n+1}.$ We denote by $\partial_\infty\mathbb{H}^{n+1}$
the boundary at infinity of $\mathbb{H}^{n+1},$ which is identified with $\R^{n}\times\{0\}.$

Han and Jiang \cite{HanJiang} studied the boundary expansions for minimal graphs in the hyperbolic space. Subsequently,
Han and Wang \cite{HanWang} studied the boundary expansions for constant mean curvature graph in the hyperbolic space.
In this paper, we will follow their ideas to study the boundary expansions of solutions to the Dirichlet problem for
constant general curvature graphs in hyperbolic space. The equations we study here are fully nonlinear, in contrast to the constant mean curvature case which is quasilinear. The biggest challenge here is to improve the lower order estimates and to work through a more complicated formal computation to find dominating terms.

In order to illustrate our ideas more clearly, in this paper, we will focus on the case when $\phi$ is curvature quotient, i.e., we let
\[\phi(\kappa[\Sigma])=\frac{H_n(\kappa)}{H_l(\kappa)}=\sigma^{n-l},\]
where
\[H_k=\frac{1}{\binom{n}{k}}\sum\limits_{1\leq j_1<\cdots<j_k\leq n}\kappa_{j_1}\cdots\kappa_{j_k}\]
is the normalized $k$-th elementary symmetric function and we set $H_0\equiv 1$. Our ideas also work for a much more general curvature setting (e.g., $H_i$, $1\leq i\leq n$).

Under our assumptions, we can rewrite equation \eqref{eq0} as follows,
\begin{align}\begin{split}\label{eq-Main}
\frac{H_n(\kappa[a^v_{ij}])}{H_l(\kappa[ a^v_{ij}])} &= \sigma^{n-l} \quad \text{ in  } \Omega,\\
f&>0 \quad \text{ in  } \Omega,\\
f&=0  \quad \text{ on  } \partial\Omega,
\end{split}
\end{align}
where  the second fundamental form of $\Sigma$ is
\begin{align*}
 a^v_{ij} &= \frac{1}{w_f}(\delta_{ij}+ f \gamma_f^{ik}f_{kl}\gamma_f^{lj}),
\end{align*}
and
\begin{align*}
w_f&= \sqrt{1+|Df|^2},\\
\gamma_f^{ij} &= \delta_{ij}-\frac{f_i f_j}{w_f(1+w_f)}.\\
\end{align*}
Note that the matrix $\{\gamma_f^{ij}\}$ is invertible with inverse
\[\gamma_{ij}=\delta_{ij}+\frac{f_if_j}{1+w_f},\]
which is the square root of $\{g_{ij}^E\}.$

 Following \cite{Lin1989Invent}, we can view the graph of $f$ locally as the graph of $u,$ whose domain is the vertical plane that intersects $\Gamma$ at a boundary point. Specifically, we fix a boundary point of $\Gamma$, say the origin,
and assume that the vector $e_n = (0,\cdots, 0, 1)$ is the exterior normal vector to $\Gamma$
at the origin. Then, with $y = (y', y_n)$, the $y'$ -hyperplane is the tangent plane of $\Gamma$
at the origin, and the boundary $\Gamma$ can be expressed in a neighborhood of the origin as a graph
of a smooth function over $\mathbb{R}^{n-1}\times\{0\}$, say
$y_n=\varphi(y').$ We now denote points in $\mathbb{R}^{n+1} = \mathbb{R}^n\times\mathbb{R}$ by $(y', y_n, t)$. We can
represent the graph of $f$ as a graph of a new function $u$ defined in terms of $(y', 0, t)$
for small $y'$ and $t$, with $t > 0$. In other words, we treat $\mathbb{R}^n = \mathbb{R}^{n-1}\times\{0\}\times\mathbb{R}$ as our
new base space and write $u(y) =u(y', t)$. Then, for some $R > 0$, u
satisfies
\be\label{fe.2}
\frac{\det\lt(\frac{t}{w}\gamma^{ik}u_{kl}\gamma^{lj}-\frac{u_t}{w}\delta_{ij}\rt)}{H_l\lt(\kappa[\frac{t}{w}\gamma^{ik}u_{kl}\gamma^{lj}-\frac{u_t}{w}\delta_{ij}]\rt)}
=\sigma^{n-l}\,\, \mbox{in $G_{r, \delta}$}
\ee
and
\be\label{fe.2'}
u=\varphi\,\,\mbox{on $B_r'$,}
\ee
where $G_{r, \delta}=\{y=(y', 0, t)| |y'|<r, 0<t<\delta\},$ $B_r'=\{y=(y', 0, 0)||y'|<r\},$
$w = \sqrt{1+|Du|^2}$, and $\gamma^{ij} = \delta_{ij}-\frac{u_i u_j}{w(1+w)}$.

We can establish formal expansions for solutions of
\eqref{fe.2} and \eqref{fe.2'} in the following form: for $k\geq n+1$,
\begin{align*}
u_k= \varphi(y^\prime) +c_1(y^\prime)t+ c_2(y^\prime)t^2+\cdots+ c_{n}(y^\prime)  t^{n} + \sum_{i=n+1}^k \sum_{j=0}^{N_i} c_{i,j}(y^\prime) t^i(\log t)^j,
\end{align*}
where $c_i, c_{i,j}$'s are coefficients only depending on coordinates $y_1,\cdots, y_{n-1}$. Moreover, $c_1$, $c_2$, $\cdots$, $c_n$, $c_{n+1,1}$ have explicit expressions in terms of $\varphi$, which will be derived in Section \ref{fe}. For $0\leq k\leq n$, we simply denote,
\begin{align*}
u_k= \varphi(y^\prime) +c_1(y^\prime)t+ c_2(y^\prime)t^2+\cdots+ c_{k}(y^\prime) t^{k},
\end{align*}
where $c_1$, $c_2$, $\cdots$, $c_k$, have explicit expressions in terms of $\varphi$.

We have following main theorems.
\begin{theorem}\label{thm-Main}
Assume in $G_{r, \delta}=\{(y^\prime, t): |y^\prime|<r, 0<t<\delta\}$, $\varphi\in C^{p, \alpha}(B_r'),$ for some $\alpha\in (0, 1).$  Let $u\in C^1(\bar{G}_{r, \delta})\cap C^{\infty}(G_{r, \delta})$ be a solution of equations \eqref{fe.2} and \eqref{fe.2'} satisfies
$
|D u|+ t|D^2 u|< C.
$
Then for any $0\leq  k \leq p$, $0<r'<r, 0<\delta'<\delta$, and $\epsilon\in (0, \alpha),$  there exists $c_i, c_{i,j}\in C^{p-i, \epsilon}(G_{r^\prime, \delta^\prime}),$ such that,
for any $\alpha\in (0,1)$, $\tau = 0, 1,\cdots, p-k,$ and any
$m = 0, 1,\cdots, k,$  we have
\be\label{eq0.1}
D^{\tau}_{y'}\partial_t^m(u-u_k)\in C^\epsilon(\bar{G}_{r', \delta'}),
\ee
and for any $(y', t)\in G_{r', \delta'},$
\be\label{eq0.2}
|D^{\tau}_{y'}\partial_t^m(u-u_k)(y', t)|\leq Ct^{k-m+\alpha}.
\ee
Here, $C$ is a constant depending only on $n, p, \alpha, r, \delta,$ the $C^0$ norm of $u$
in $G_{r, \delta},$ and the $C^{p, \alpha}$ norm of $\varphi$ in $B'_r.$
\end{theorem}

We also have the local convergence theorem,
\begin{theorem}\label{thm-MainConv}
Assume the same assumptions as in Theorem \ref{thm-Main}. In addition, if $\varphi$ is analytic in $y^\prime$, then the series $\{u_k\}$ with logarithmic terms derived from the boundary expansion converges to $u$ uniformly in $G_{r^\prime, \delta^\prime}$, for any $0<r^\prime<r, 0<\delta^\prime<\delta$. Furthermore, $u$ is analytic in \begin{align*}
y^\prime, t, t\log t
\end{align*}
 for $(y^\prime, t)\in \bar{G}_{r^\prime, \delta^\prime}\cup \{(y^\prime, 0): |y^\prime|< r^\prime\}$.
\end{theorem}

\section{Preliminaries}\label{sec-Exp}
The main purpose of this section is to derive the formula for hyperbolic principal curvatures and to construct explicit solutions of constant curvature.
These explicit solutions will serve as barrier functions later. Readers who are familiar with hyperbolic geometry can skip this section.
\subsection{Formulas for hyperbolic principal curvatures}
\label{fml}
Let $\Sigma$ be a hypersurface in $\mH^{n+1}.$ We shall use $g$ and $\nabla$ to denote the induced
hyperbolic metric and Levi-Civita connections on $\Sigma$, respectively. Since $\Sigma$ also can
be viewed as a submanifold of $\R^{n+1}$, we will usually identify a geodesic quantity
with respect to the Euclidean metric by adding a superscript `` $E$ " over the corresponding
hyperbolic quantity. For instance, $g^E$ denotes the induced metric on $\Sigma$ from $\R^{n+1}$
and $\nabla^E$ is its Levi-Civita connection.

Let $(z_1, \cdots, z_n)$ be local coordinates and
\[\tau_i=\frac{\partial}{\partial z_i}, \,\, i=1, \cdots, n.\]
The hyperbolic and Euclidean metrics of $\Sigma$ are given by
\[g_{ij}=\lt<\tau_i, \tau_j\rt>_H,\,\,g^E_{ij}=\tau_i\cdot\tau_j=x_{n+1}^2g_{ij};\]
while the second fundamental forms are
\[h_{ij}=\lt<\nabla_{\tau_i}\tau_j, \nu_H\rt>_H,\,\,h^E_{ij}=\nabla^E_{\tau_i}\tau_j\cdot\nu,\]
where $\nu$ is the Euclidean normal to $\Sigma$ and $\nu_H=\frac{\nu}{x_{n+1}}$ is the hyperbolic normal.
The following relations are well known:
\be\label{fml.1}h_{ij}=\frac{1}{x_{n+1}}h^E_{ij}+\frac{\nu^{n+1}}{x_{n+1}^2}g^E_{ij}\ee
and
\be\label{fml.2}\kappa_i=x_{n+1}\kappa^E_i+\nu^{n+1},\ee
where $\nu^{n+1}=\nu\cdot e_{n+1}.$

Now suppose $\Sigma$ can be represented as the graph of a function $f\in C^2(\Omega),$
$f > 0$, in a domain $\Omega\subset\R^n:$
\[\Sigma = \{(x, f(x))\in\R^{n+1}: x\in\Omega\}.\]
In this case we take $\nu$ to be the upward (Euclidean) unit normal vector field to $\Sigma,$ then we have
\[\nu=\frac{(-Df, 1)}{w_f},\,\,w_f=\sqrt{1+|Df|^2}.\]
The Euclidean metric and second fundamental form of $\Sigma$ are given respectively by
\[g^E_{ij}=\delta_{ij}+f_if_j\]
and
\[h^E_{ij}=\frac{f_{ij}}{w_f}.\]
The Euclidean principal curvatures $\kappa^E[\Sigma]$ are the eigenvalues of the
symmetric matrix $A^E[f] =\{a^E_{ij}\}$:
\[a^E_{ij}=\frac{1}{w_f}\gamma^{ik}f_{kl}\gamma^{lj},\]
where $\gamma^{ij}=\delta_{ij}-\frac{f_if_j}{w_f(1+w_f)}.$
From \eqref{fml.2}, we can see that the
hyperbolic principal curvatures $\kappa_i[\Sigma]$ are the eigenvalues of the matrix
$A^v[f]=\{a^v_{ij}\}:$
\[ a^v_{ij} = \frac{1}{w_f}(\delta_{ij}+ f \gamma_f^{ik}f_{kl}\gamma_f^{lj}).\]

After rewriting this $\Sigma$ locally as a graph of a new function $u$ over the vertical plane, i.e.
\[\Sigma=\{(y', u(y', t), t)| y=(y', t)\in\R^{n-1}\times\{0\}\times\R\},\]
we have
\[\nu=\frac{(-u_{y'}, 1, -u_t)}{w} \,\,\mbox{and $h_{ij}=\frac{u_{ij}}{w}$}.\]
Therefore, under this coordinates
\[a_{ij}^v=\frac{t}{w}\gamma^{ik}u_{kl}\gamma^{lj}-\frac{u_t}{w}\delta_{ij}.\]

\subsection{Explicit Solutions on Balls}
\label{esb}
In this subsection, we will construct sub and super solutions to the equations \eqref{fe.2} and \eqref{fe.2'}.
Our construction based on following well known fact: Let $B_R(a)$ be a ball of
radius $R$ centered at $a = (a', \pm\sigma R)$ in $\mathbb{R}^{n+1}$ where $a'\in\mathbb{R}^n$,
then by equation \eqref{fml.2}, $S =B_R(a)\cap\mathbb{H}^{n+1}$
has constant hyperbolic principal curvature $\sigma$ with respect to its upward normal.

Let's recall a comparison Lemma in \cite{GS00}:
\begin{lemma}\label{lem2.1}
Let $B_1$ and $B_2$ be balls in $\mathbb{R}^{n+1}$ of radius $R$ centered at $a =(a',-\sigma R)$
and $b = (b', \sigma R),$ respectively. Let $\Sigma$ be a hypersurface satisfies equation \eqref{eq-Main}.\\
(i) If $\partial\Sigma\subset B_1,$ then $\Sigma\subset B_1$.\\
(ii) If $B_1\cap\mathbb{R}^n\times\{0\}\subset\Omega$, then $B_1\cap\Sigma =\emptyset .$\\
(iii) If $B_2\cap\Omega=\emptyset$, then $B_2\cap \Sigma =\emptyset.$
\end{lemma}

For a given boundary $\Gamma=\partial\Omega\subset\partial_\infty\mH^{n+1},$ we say that $\Gamma$ satisfies
the \textit{uniform interior (resp. exterior) local ball condition},
if there exists some $\alpha>0$ such that for all $Q\in\Gamma,$
there exists a ball $B^n_\alpha\subset\bar{\Omega}$ $(resp. B^n_\alpha\subset\bar{\Omega}^c)$, and $\partial B^n_\alpha\cap\partial\Omega=\{Q\}.$
In this paper, we always assume $\Gamma$ satisfies the uniform interior/exterior local ball condition.
In particular, let's denote $x=(x', x_n)$ the coordinates in $\mathbb{R}^{n+1}$, where
\begin{align*}
x^\prime = (x_1, x_2, \cdots, x_{n-1}).
\end{align*}
Assume $P=(x^\prime_P, s, 0)\in \mathbb{R}^n \times \{0\}$ and $B^n_\alpha(P)\subset\mathbb{R}^n$ be an Euclidean ball centered at $P$ with radius $\alpha$ in $\partial_\infty\mathbb{H}^{n+1}$.
The explicit solution over domain $B^n_\alpha(P)$ of equation \eqref{eq-Main} is the following
\begin{align*}
f_\alpha(x)=  \sqrt{R^2 - |x^\prime-x_P^\prime |^2-|x_n- s|^2} -\sigma R,
\end{align*}
where $R>0$ satisfies $\alpha^2+ \sigma^2 R^2= R^2$. The graph of $f_\alpha$ is a portion of the sphere centered at $(x^\prime_p, s,-\sigma R)$ with radius $R$ in $\mathbb{R}^{n+1}$.

Now, if the ball $B^n_\alpha(P)\subset\bar{\Omega}$ is tangential to a boundary point $Q= (x_Q^\prime, \varphi(x_Q^\prime)) \in            \partial \Omega$ in the neighborhood of origin $O \in \partial \Omega$.  Then, we can see that the inner normal vector to $\partial\Omega$ at $Q$ is $N= (\frac{D\varphi}{\sqrt{1+|D\varphi|^2}}, \frac{-1}{\sqrt{1+|D\varphi|^2}}, 0)$. Let $P= Q + \alpha N$, we have
\begin{align*}
f_\alpha(x)&=  \sqrt{R^2 - \left|x^\prime-x_Q^\prime -\frac{\alpha D\varphi}{\sqrt{1+|D\varphi|^2}}(x_Q^\prime)\right|^2-\left|x_n- \varphi(x_Q^\prime) + \frac{\alpha}{\sqrt{1+|D\varphi|^2}}(x_Q^\prime)\right|^2}\\
&\qquad -\sigma R.
\end{align*}
Then the corresponding $u_\alpha$ on the vertical plane $\{x_n=0\}$ is
\begin{align*}
u_R(y)=u_R(y', t)&=\sqrt{R^2- (t+\sigma R)^2 - \left|y^\prime-y_Q^\prime +\frac{\alpha D\varphi}{\sqrt{1+|D\varphi|^2}}(y_Q^\prime)\right|^2}\\
&\qquad +\left( -\frac{\alpha}{\sqrt{1+|D\varphi|^2}}(y_Q^\prime)+ \varphi (y_Q^\prime) \right).
\end{align*}
If we only consider at points where $y^\prime= y^\prime_Q$, we have
\begin{align*}
u_R(y)&=\varphi+\sqrt{R^2- (t+\sigma R)^2 - \left|\frac{\alpha D\varphi}{\sqrt{1+|D\varphi|^2}}\right|^2} -\frac{\alpha}{\sqrt{1+|D\varphi|^2}},
\end{align*}
which equals $\varphi$ when $t=0$. Furthermore, it's easy to check, in this case,
\begin{align}\label{eq-SolOnBall}
u_R(y)= \varphi  - \sqrt{\frac{\sigma^2 (1+|D\varphi|^2 )}{1-\sigma^2} }t+ t^2 F(y, R, \alpha)
\end{align}
where $F(y, R, \alpha)$ is uniformly bounded.
By Lemma \ref{lem2.1}, we conclude that $u_R\leq u$ in the neighborhood of the origin.

Similarly, if the ball $B^n_\alpha(P)\subset\mathbb{R}^n\backslash\Omega$ is tangential to $\partial\Omega$
from outside, then in a small neighborhood of the origin, we get
\begin{align}\begin{split}\label{eq-ExpBall}
u^R(y)&=-\sqrt{R^2- (t-\sigma R)^2 - \left|y'-y'_Q+\frac{\alpha D\varphi}{\sqrt{1+|D\varphi|^2}}(y^\prime)\right|^2}\\
&\qquad + \frac{\alpha}{\sqrt{1+|D\varphi|^2}}(y^\prime)+ \varphi(y^\prime).
\end{split}
\end{align}
Moreover, when $y'=y_Q'$ we can see that
\be\label{eq-SolOnBall'}
u^R(y)= \varphi-\sqrt{\frac{\sigma^2 (1+|D\varphi|^2 )}{1-\sigma^2} }t+ t^2 F'(y, R, \alpha),
\ee
where $F'(y, R, \alpha)$ is uniformly bounded.
By Lemma \ref{lem2.1}, we conclude that $u^R\geq u$ in the neighborhood of the origin.

$u_R$ and $u^R$ will serve as barrier functions in the proof of Lemma \ref{lem-T2}.

\section{Formal Expansions}
\label{fe}
In this section, we derive expansions for solution to the equations \eqref{fe.2} and \eqref{fe.2'}.
For our convenience, we will rewrite equation \eqref{fe.2} as follows:
\be\label{fe.3}
\det(t\gamma^{ik}u_{kl}\gamma^{lj}-u_t\delta_{ij})=\sigma^{n-l}w^{n-l}H_l(\kappa[t\gamma^{ik}u_{kl}\gamma^{lj}-u_t\delta_{ij}]).
\ee
We denote equation \eqref{fe.3} by $Q(u)=\bQ(D^2u, Du, t)=0$, where
\begin{align*}
Q(u)=\bQ(D^2u, Du, t)= \det(t\gamma^{ik}u_{kl}\gamma^{lj}-u_t\delta_{ij})-\sigma^{n-l}w^{n-l}H_l(\kappa[t\gamma^{ik}u_{kl}\gamma^{lj}-u_t\delta_{ij}]).
\end{align*}

Now, set $u_N=\varphi+\sum_{i=1}^Nc_it^i,$ where $\varphi$ as explained in Section \ref{sec0}, only relies on the tangential coordinates and its graph defines the boundary locally. Our goal is to solve for the tangential coefficients $c_i$'s inductively such that $Q(u_N)= O(t^{N})$.

 By differentiating $u_N$ twice we get,
\be\label{fe.01}Du_N=(D\varphi+ c_1 Dt) +\sum_{i=2}^{N}(Dc_{i-1} +ic_{i} Dt)t^{i-1}+ Dc_N \cdot t^N,\ee
and
\be\label{fe.02}
D^2u_N=D^2 \varphi+\sum_{i=2}^N  i(i-1)c_it^{i-2}+\cdots.
\ee
In $``\cdots"$, when $c_i$ appears, it's of the form $Dc_i, D^2c_i$ and it has a factor $t^{i-1},$ $t^i$ respectively.
Substituting equations \eqref{fe.01} and \eqref{fe.02} into equation \eqref{fe.3} we get the following.

First, by calculating the coefficients of the $O(1)$ term, we get
\be\label{fe.4}
\lhs :=(-c_1)^n,
\ee
and
\be\label{fe.5}
\rhs :=\sigma^{n-l}(1+|D\varphi|^2+c_1^2)^{\frac{n-l}{2}}(-c_1)^l.
\ee
Therefore we obtain
\[(-c_1)^{n-l}=\sigma^{n-l}(1+|D\varphi|^2+c_1^2)^{\frac{n-l}{2}},\]
which yields
\be\label{fe.6}
c_1=-\sqrt{\frac{\sigma^2(1+|D\varphi|^2)}{1-\sigma^2}},
\ee
here $c_1<0$ because of our choice of orientation. Moreover, we have $Q(\varphi+c_1t)=O(t).$

Next, we will look at terms of order $O(t).$ We collect those containing $c_2$ on the $\lhs$ of \eqref{fe.3},
\be\label{fe.7}
\begin{aligned}
\lhs :&=[-2c_2t\delta_{ii}+t\gamma^{it}\cdot 2c_2 \cdot\gamma^{ti}](-c_1)^{n-1}\\
&=[-2c_2tn+2c_2t(1-\frac{u_t^2}{w^2})](-c_1)^{n-1}\\
&=-2c_2t(-c_1)^{n-1}\lt(n-1+\frac{u_t^2}{w^2}\rt)\\
&=-2c_2t(-c_1)^{n-1}\lt(n-\frac{1+|D\varphi|^2}{1+|D\varphi|^2+c_1^2}\rt)\\
&=-2c_2t(-c_1)^{n-1}\lt(n-(1-\sigma^2)\rt).
\end{aligned}
\ee
The rest $O(t)$ terms on the $\lhs$ relies on $\varphi, c_1$ and their derivatives.

On the $\rhs$, we have
\begin{align*}
w^{n-l}&=(1+|Du|^2)^\frac{n-l}{2}\\
&= (1+|D\varphi|^2+c_1^2 +2 D\varphi Dc_1 t +4c_1 c_2 t+\cdots)^\frac{n-l}{2}\\
&= (1+|D\varphi|^2+c_1^2)^\frac{n-l}{2} + (1+|D\varphi|^2+c_1^2)^\frac{n-l-2}{2} \cdot \frac{n-l}{2} (4c_1 c_2 t  +2 D\varphi Dc_1 t)\\
&\qquad+\cdots,
\end{align*}
where "$\cdots$" denotes higher order terms than $t$. Notice $H_l(\kappa[A_{ij}])$ equals the sum of all $l \times l$ submatrices of $A$ that have same rows and columns. Hence,
\begin{align*}
H_l(\kappa[t\gamma^{ik}u_{kl}\gamma^{lj}-u_t\delta_{ij}])&=  (-c_1)^l+(-c_1)^{l-1}\cdot (-2c_2 lt+ t\gamma^{it}u_{tt}\gamma^{ti}\frac{l}{n})+\cdots\\
&=   (-c_1)^l+(-c_1)^{l-1}\cdot \lt[-2c_2lt+ 2c_2 t \cdot (1-\frac{u_t^2}{w^2}) \frac{l}{n}\rt]+ \cdots,
\end{align*}
where "$\cdots$" denotes order $t$ terms without $c_2$ factor and higher order terms. So the $O(t)$ terms containing $c_2$ in the $\rhs$ of \eqref{fe.3} are
\be\label{fe.8}
\begin{aligned}
\rhs :&=(1+|D\varphi|^2+c_1^2)^{\frac{n-l}{2}}\sigma^{n-l}\lt[-2c_2lt+ 2c_2 t \cdot (1-\frac{u_t^2}{w^2}) \frac{l}{n}\rt](-c_1)^{l-1}\\
&\qquad+(4c_1c_2t)\cdot\frac{n-l}{2}\sigma^{n-l}(-c_1)^l(1+|D\varphi|^2+c_1^2)^\frac{n-l-2}{2}\\
&=(-c_1)^{n-1}\lt[-2c_2tn+2c_2t \cdot(1-\sigma^2)\rt]\cdot\frac{l}{n}\\
&\qquad-2(n-l)c_2t (-c_1)^{n-1}\sigma^2.
\end{aligned}
\ee

Combining \eqref{fe.7} and \eqref{fe.8} we get
\be\label{fe.9}
\begin{aligned}
&\qquad \lhs- \rhs\\
&=-2c_2t(-c_1)^{n-1}\lt(n-(1-\sigma^2)\rt)\\
&\qquad +2c_2t(-c_1)^{n-1}(n-(1-\sigma^2))\cdot\frac{l}{n}+2(n-l)c_2t (-c_1)^{n-1}\sigma^2\\
&=2c_2t(-c_1)^{n-1}\lt(1-\sigma^2 \rt) \frac{n-l}{n}(1-n)
\end{aligned}
\ee
Therefore
\begin{align*}
Q(u_2)&= (2c_2(-c_1)^{n-1}\lt(1-\sigma^2 \rt) \frac{n-l}{n}(1-n)t\\
&\qquad+ F_2(x, \varphi,  D\varphi, D^2\varphi, c_1, D^i c_1))t +O(t^2),
\end{align*}
 we can solve for $c_2$ such that $Q(u_2)=O(t^2)$.

Finally, we consider the terms of order $O(t^{k-1})$. Note that the terms containing $c_k$ have at least a $t^{k-1}$ factor, and those containing derivatives of $c_k$ have at least a $t^k$ factor. We collect all $O(t^{k-1})$ terms containing $c_k$.
\be\label{fe.10}
\begin{aligned}
\lhs :&=\lt\{-kc_kt^{k-1}\delta_{ii}+t\gamma^{it}[k(k-1)c_kt^{k-2}]\gamma^{ti}\rt\}(-c_1)^{n-1}\\
&=(-c_1)^{n-1}kc_k t^{k-1}\lt[-n+(k-1)\lt(1-\frac{c_1^2}{1+|D\varphi|^2+c_1^2}\rt)\rt];\\
&=(-c_1)^{n-1}kc_k t^{k-1} (-n+(1-\sigma^2) (k-1));\\
\end{aligned}
\ee

On the $\rhs$, similarly we have
\begin{align*}
w^{n-l}&=(1+|Du|^2)^\frac{n-l}{2}\\
&= (1+|D\varphi|^2+c_1^2 +\cdots+2kc_k t^{k-1}c_1+\cdots)^\frac{n-l}{2}\\
&= (1+|D\varphi|^2+c_1^2)^\frac{n-l}{2} +\cdots+ (1+|D\varphi|^2+c_1^2)^\frac{n-l-2}{2} \cdot \frac{n-l}{2} (2kc_k t^{k-1}c_1)\\
&\qquad+\cdots,
\end{align*}
where "$\cdots$" denotes higher order terms than $O(1)$ which either does not contain $c_k$ or has higher order than $t^{k-1}$. And
\begin{align*}
H_l(\kappa[t\gamma^{ik}u_{kl}\gamma^{lj}-u_t\delta_{ij}])&=  (-c_1)^l+ (-c_1)^{l-1} \cdot(-kc_k lt^{k-1}+ t\gamma^{it}u_{tt}\gamma^{ti})+\cdots\\
&=   (-c_1)^l+(-c_1)^{l-1}\cdot \lt[-kc_klt^{k-1}+ k(k-1)c_k t^{k-1} \cdot (1-\frac{u_t^2}{w^2}) \frac{l}{n}\rt]\\
&\qquad+ \cdots,
\end{align*}
where "$\cdots$" denotes higher order terms than $O(1)$ which either does not contain $c_k$ or has higher order than $t^{k-1}$. So the $O(t^{k-1})$ terms containing $c_k$ in the $\rhs$ of \eqref{fe.3} are
\be\label{fe.11}
\begin{aligned}
\rhs :
&=(1+|D\varphi|^2+c_1^2)^{\frac{n-l}{2}}\sigma^{n-l}\lt[-kc_klt^{k-1}+ k(k-1)c_k t^{k-1} \cdot (1-\frac{u_t^2}{w^2}) \frac{l}{n}\rt](-c_1)^{l-1}\\
&\qquad+(2kc_1c_kt^{k-1})\cdot\frac{n-l}{2}\sigma^{n-l}(-c_1)^l(1+|D\varphi|^2+c_1^2)^\frac{n-l-2}{2}\\
&=(-c_1)^{n-1}\lt[-kc_kt^{k-1}  n+k(k-1)c_kt^{k-1}  \cdot(1-\sigma^2)\rt]\cdot\frac{l}{n}\\
&\qquad-k(n-l)c_kt^{k-1} (-c_1)^{n-1}\sigma^2.
\end{aligned}
\ee

Combining \eqref{fe.10} and \eqref{fe.11} we get
\be\label{fe.12}
\begin{aligned}
&\qquad \lhs- \rhs\\
&=-kc_kt^{k-1}(-c_1)^{n-1}(n-(1-\sigma^2) (k-1))\\
&\qquad +kc_kt^{k-1}(-c_1)^{n-1}(n-(1-\sigma^2)(k-1) )\cdot\frac{l}{n}+k(n-l)c_kt^{k-1}(-c_1)^{n-1}\sigma^2\\
&=kc_kt^{k-1}(-c_1)^{n-1}\lt(1-\sigma^2 \rt) \frac{n-l}{n}(k-1-n)
\end{aligned}
\ee
Therefore
\begin{align*}
Q(u_k)&= [ c_k   k(-c_1)^{n-1}\lt(1-\sigma^2 \rt) \frac{n-l}{n}(k-1-n)\\
&\qquad+ F_k(x, \cdots)]t^{k-1} +O(t^k),
\end{align*}
where $F_k$ is a function of $x, \varphi, c_1,\cdots, c_{k-1}$ and their tangential directives.
 We can solve for $c_k$ for $1\leq k \leq n$ such that $Q(u_k)=O(t^k)$.

 Above procedure fails at $k=n+1$ since the coefficient of $c_{n+1}$ would be $0$. We introduce a term $c_{n+1,1}t^{n+1}\log t$, and assume
\begin{align}
u^* =u_n+ c_{n+1,1}t^{n+1}\log t.
 \label{eq-VStar}
\end{align}
 Then we have
 \begin{lemma}
 \label{felm.1}
 Let $p\geq n+3,$ then for any $\varphi\in C^p(B_r'),$ there exist $c_i\in C^{p-i}(B_r'),$ $1\leq i\leq n$ and $c_{n+1, 1}\in C^{p-n-1}(B_r')$
 such that $u^*$ defined in \eqref{eq-VStar} satisfies
 \be\label{fe.13}
 \lt|Q(u^*)\rt|\leq Ct^{n+1}(-\log t),
 \ee
 where $C$ is a positive constant depending only on $n$ and the $C^{n+3}$ norm of $\varphi.$
 \end{lemma}
 \begin{proof}
 We already derived $|Q(u_n)|\leq Ct^n$ with uniquely solved coefficients $c_0, \cdots, c_n$. For $c_{n+1,1}$, on the $\lhs$ of \eqref{fe.3}, the $O(t^n)$ terms containing $c_{n+1,1}$ are
 \begin{align*}
 (-c_1)^{n-1}c_{n+1,1}(-n+(1-\sigma^2) \cdot 2 (n+1))t^n
 \end{align*}
 on the $\rhs$ of \eqref{fe.3}, the $O(t^n)$ terms containing $c_{n+1,1}$ is
 \begin{align*}
 &(-c_1)^{n-1}\lt[-c_{n+1,1 }t^{n}  n+2(n+1)c_{n+1,1}t^{n}  \cdot(1-\sigma^2)\rt]\cdot\frac{l}{n}\\
&\qquad-(n-l)c_{n+1,1}t^{n} (-c_1)^{n-1}\sigma^2.
 \end{align*}
 The subtraction of these two are
 \begin{align*}
c_{n+1,1}t^{n}(-c_1)^{n-1}\lt(1-\sigma^2 \rt) \frac{n-l}{n}(n+2).
 \end{align*}

Therefore, by requiring the coefficient of $t^n$ to be zero in $Q(u^*)$ we have
 \[ (-c_1)^{n-1}\lt(1-\sigma^2 \rt) \frac{n-l}{n}(n+2) c_{n+1, 1}=F_{n+1, 1}(\varphi, c_1, c_2, \cdots, c_n, \sigma).\] Since $t^{n+1}(-\log t)^i$ for $i\geq 2$ does not appear in the formal computation,
 \[\lt|Q(u^*)\rt|\leq Ct^{n+1}(-\log t).\]
 \end{proof}

 \bigskip
 \section{Estimates of Local terms}
 \label{es}
\quad In this section, we derive an estimate for an expansion involving all local terms by the maximum principle.
We denote $y=(y', t)$ for points in $\R^n,$ with $y_n=t$, and set, for any $r, \delta>0,$
 \[G_{r,\delta}:=\{(y', t): |y'|<r, 0<t<\delta\}.\]
We consider
\be\label{eq4.1}
 \det(t\gamma^{ik}u_{kl}\gamma^{lj}-u_t\delta_{ij})-\sigma^{n-l}w^{n-l}H_l(\kappa[t\gamma^{ik}u_{kl}\gamma^{lj}-u_t\delta_{ij}])=0\,\,\mbox{in $G_{r, \delta}$},
\ee
and
\be\label{eq4.2}
u=\varphi\,\,\mbox{on $B'_r$.}
\ee

First, we derive a decay estimate by applying Lemma \ref{lem2.1}.

\begin{lemma}
\label{lem-T2}
Assume $\varphi(y^\prime)\in C^2(B_r')$ and let $u\in C(\bar{G}_{r, \delta})\cap C^\infty(G_{r,\delta})$ be a solution of the equations \eqref{eq4.1} and \eqref{eq4.2}. Then, for any $(y', t)\in G_{r_1, \delta_1},$ $0<r_1<r,$ $0<\delta_1<\delta,$ we have
\be\label{es.11}
|u-\varphi-c_1(y')t|\leq Ct^2,
\ee
where $c_1(y^\prime)=-\sqrt{\frac{\sigma^2(1+|D \varphi|^2)}{1-\sigma^2}},$ and $C=C(r, \delta, r_1, \delta_1, \partial \Omega)$.
\end{lemma}
\begin{proof}
First, we show $u- \varphi- c_1(y^\prime)t \geq - C t^2$ for some $C>0$. For any point $y=(y', t)$ in $G_{r_1, \delta_1}$, we can find a boundary point $Q \in \partial G_{r_1, \delta_1} \cap \{t=0\}$ such that $y^\prime(Q)= y^\prime$. Then by the uniform interior local ball condition, there exists an interior ball $B_\alpha(P)$ tangent to $\partial \Omega$ at $Q$. By Lemma \ref{lem2.1} we have $f>f_\alpha$, and hence
\begin{align*}
u> u_R= \varphi+c_1(y^\prime)t+ t^2 F(y, R, \alpha)
\end{align*}
as shown at \eqref{eq-SolOnBall}.

For the other direction $u- \varphi- c_1(y^\prime)t \leq C t^2,$ similarly, there exists an exterior ball
$B_\alpha(P')$ tangent to $\partial\Omega$ at $Q.$ By Lemma \ref{lem2.1} and equation \eqref{eq-SolOnBall'}
we have
\begin{align*}
u< u^R= \varphi+c_1(y^\prime)t+ t^2 F'(y, R, \alpha).
\end{align*}
Therefore, \eqref{es.11} holds.
\end{proof}

Next, we prove an estimate for an expansion of solutions involving all the local terms by the maximum principle.
\begin{theorem}
\label{th4.1}
Assume $0<\delta<\frac{1}{2},$ $\varphi\in C^{n+3}(B_{\sqrt{\delta}}')$ and $u\in C(\bar{G}_{\sqrt{\delta}, \delta})\cap C^{\infty}(G_{\sqrt{\delta}, \delta})$
is a solution of equations \eqref{eq4.1} and \eqref{eq4.2}. Then for any $0<r_1<\sqrt{\delta},$ $0<\delta_1<\delta$
we have
\be\label{eq4.3}
|u-u^*|\leq Ct^{n+1}.
\ee
where $C=C(|\varphi|_{C^{n+3}(B'_{\sqrt{\delta}})}, \partial\Omega, \sigma).$
\end{theorem}
\begin{proof}
$\bf{Step \, 1}$.
Recall that $u$ satisfies
\begin{align*}
\bar{Q}(D^2 u, Du, t)=\det(t\gamma^{ik}u_{kl}\gamma^{lj}-u_t\delta_{ij})-\sigma^{n-l}w^{n-l}H_l(\kappa[t\gamma^{ik}u_{kl}\gamma^{lj}-u_t\delta_{ij}])=0.
\end{align*}
We will construct supersolution and subsolution for equations \eqref{eq4.1} and \eqref{eq4.2}.
Set
\[\psi=A\lt((|y'|^2+t)^{n+1}-(|y'|^2+t)^q\rt),\]
where $n+1<q<n+2,$ without loss of generality we let $q=n+1+\frac{1}{2}.$
It's easy to see that on $\bar{G}_{\sqrt{\delta}, \delta}\cap\{t=0\}$,
\[u\leq u^*+\psi.\]
By Lemma \ref{lem-T2}, we can choose $A>C_1\delta^{1-n}$ such that
\[A(\delta^{n+1}-\delta^{n+1+\frac{1}{2}})>C\delta^2.\]
Then on the set $\bar{G}_{\sqrt{\delta}, \delta}\cap\lt(\{|y'|=\sqrt{\delta}\}\times\{0<t<\delta\}\rt)$
 we have $u^*+\psi>u.$
 Therefore, we have
 \[u<u^*+\psi\,\,\mbox{on $\partial G_{\sqrt{\delta}, \delta}.$}\]
 We will prove that there exists a $C_0=C_0(|\varphi|_{C^{n+3}(B'_{\sqrt{\delta}})}, \partial\Omega, \sigma)>C_1,$ such that
 $u\leq u^*+\psi$ in $G_{\sqrt{\delta}, \delta}.$

$\bf{Step  \,2}$. We will show that $Q(u^*+\psi)\leq 0$ in $G_{\sqrt{\delta}, \delta}.$
A straightforward calculation yields
\be\label{eq4.4}
\begin{aligned}
&\qquad Q(u^*+\psi)-Q(u^*)\\
&=\int_0^1\frac{\partial\bar{Q}}{\partial u_{ij}}(u^*+s\psi)ds\psi_{ij}
+\int_0^1\frac{\partial\bar{Q}}{\partial u_i}(u^*+s\psi)ds\psi_i.
\end{aligned}
\ee
Now let's denote $\theta:=|y'|^2+t,$ and we will use the Greek letters $\alpha, \beta$ for indices that less than $n;$ use English letters $i, j, k, l, s$ for indices less than or equal to $n.$ By differentiating $\psi$ we obtain,
\begin{align}
\psi_\alpha&=A(2(n+1)\theta^ny_{\alpha}-2q\theta^{q-1}y_\alpha), \label{eq4.5}\\
\psi_{\alpha\beta}&=A[4(n+1)n\theta^{n-1}y_\alpha y_\beta+2(n+1)\theta^n\delta_{\alpha\beta} \label{eq4.6}\\
&\qquad-4q(q-1)\theta^{q-2}y_\alpha y_\beta-2q\theta^{q-1}\delta_{\alpha\beta}], \nonumber\\
\psi_t&=A[(n+1)\theta^n-q\theta^{q-1}],\label{eq4.7}\\
\psi_{\alpha t}&=A[2(n+1)n\theta^{n-1}y_\alpha-2q(q-1)\theta^{q-1}y_\alpha]\label{eq4.8},\\
\end{align}
and
\be\label{eq4.9}
\psi_{tt}=A[n(n+1)\theta^{n-1}-q(q-1)\theta^{q-2}].
\ee
Denote
\[M_{ij}(u)=t\gamma^{ik}u_{kl}\gamma^{lj}-u_t\delta_{ij}=m_{ij}(u)-u_t\delta_{ij},\]
and
\[v^*=u^*-\varphi-c_1(y')t,\]
then by a straightforward calculation we get
\be\label{eq4.10}
M_{ij}(u^*+s\psi)=(-c_1-s\Lambda)\delta_{ij}+tA_1(D^2v^*, \frac{Dv^*}{t}, y',t, s),
\ee
and
\be\label{eq4.11}
w^2(u^*+s\psi)=1+|D\varphi|^2+c_1^2+s^2\sum_\alpha|\psi_\alpha|^2+tA_2(\frac{Dv^*}{t}, y', t, s)+s\Lambda\mathbb{P}_1(\Lambda)
\ee
where $\Lambda=\psi_t,$ and $\mathbb{P}_k(\Lambda)$ denotes polynomial of degree $k$ with respect to $\Lambda.$
In the following, we denote $A_i's,$ $i\in\mathbb{Z}^+$ as functions that are smooth in their arguments.
One can check that essentially $A_i's$ can be viewed as polynomials of $D^2v^*, \frac{Dv^*}{t}, t$
with coefficients in $C^{\infty}(G_{\sqrt{\delta}, \delta}).$

$\bf{Step \, 3}.$
Let
\[\bQ(D^2u, Du, t)=\Phi(M_{ij})=\det M_{ij}-\sigma^{n-l}w^{n-l}H_l(\kappa[M_{ij}]).\]
By equation \eqref{eq4.10} we get
\be
\begin{aligned}
\label{eq4.12}
\frac{\partial\bQ}{\partial u_{ij}}(u^*+s\psi)&=\frac{\partial\Phi}{\partial M_{kl}}\frac{\partial M_{kl}}{\partial u_{ij}}\\
&=t\gamma^{ki}\frac{\partial\Phi}{\partial M_{kl}}\gamma^{lj}\\
&=t\gamma^{ik}\gamma^{lj}\delta_{kl}\lt[\frac{n-l}{n}|c_1|^{n-1}-s\Lambda\mathbb{P}_{n-2}(\La)\rt]+t^2 A_3+tO(|\tL|^2),\\
\end{aligned}
\ee
where $|\tL|^2=\sum_\alpha|\psi_\alpha|^2.$

Recall that we denote $m_{ij}(u)=t\ga^{ik}u_{kl}\ga^{lj},$
and $\ga^{ik}=\delta_{ik}-\frac{u_iu_k}{w(1+w)}.$ We will also denote
$\{\ga_{ij}\}=\{\ga^{ij}\}^{-1},$ and it's easy to see that $\ga_{ij}=\delta_{ij}-\frac{u_iu_j}{1+w}.$
In the following, we will compute $\frac{\partial m_{ij}}{\partial u_r}.$

Since
 \[t\ga^{ik}u_{kl}\ga^{lj}\ga_{kl}=t\ga^{ik}u_{kl}=m_{ik}\ga_{kl},\]
we have
\[\frac{\partial m_{ij}}{\partial u_r}=2tu_{kl}\ga^{ik}\frac{\partial\ga^{lj}}{\partial u_r}
=2m_{ik}\ga_{kl}\frac{\partial\ga^{lj}}{\partial u_r}=-2m_{ik}\ga^{lj}\frac{\partial\ga_{kl}}{\partial u_r}.\]
Moreover,
\[\frac{\partial\ga_{kl}}{\partial u_r}=\frac{u_k\delta_{lr}+u_l\ga^{kr}}{1+w},\]
thus we get
\be\label{eq4.13}
\frac{\partial m_{ij}}{\partial u_r}=-2\frac{m_{ik}}{1+w}\lt(u_k\ga^{rj}+\frac{u_j\ga^{kr}}{w}\rt).
\ee
Therefore
\be\label{eq4.14}
\begin{aligned}
\frac{\partial\bQ}{\partial u_r}&=\frac{\partial\Phi}{\partial M_{ij}}\frac{\partial M_{ij}}{\partial u_r}
-(n-l)\sigma^{n-l}w^{n-l-1}H_l\frac{\partial w}{\partial u_r}\\
&=\lt\{\lt[\frac{n-l}{n}|c_1|^{n-1}-s\La\mathbb{P}_{n-2}(\La)\rt]\delta_{ij}+tA_4+O(|\tL|^2)\rt\}\\
&\qquad\times\lt[-2\frac{m_{ik}}{1+w}\lt(u_k\ga^{rj}+\frac{u_j\ga^{kr}}{w}\rt)-\delta_{rt}\delta_{ij}\rt]\\
&\qquad-(n-l)\sigma^{n-l}w^{n-l-2}H_l(u_r^*+s\psi_r).\\
\end{aligned}
\ee
We can see that
\be\label{eq4.15}
\frac{\partial\bQ}{\partial u_{tt}}=t(1-\sigma^2)\lt[\frac{n-l}{n}|c_1|^{n-1}-s\La\mathbb{P}_{n-2}(\La)\rt]+t^2A_5+tO(|\tL|^2),
\ee
and
\be\label{eq4.16}
\frac{\partial\bQ}{\partial u_t}=-(n-l)(1-\sigma^2)|c_1|^{n-1}+s\La\mathbb{P}_{n-2}(\La)+tA_6+O(|\tL|^2).
\ee
Furthermore, one can check that
\be\label{eq4.17}
\lt|\frac{\partial\bQ}{\partial u_{\alpha j}}\rt|\leq tD_1
\ee
and
\be\label{eq4.18}
\lt|\frac{\partial\bQ}{\partial u_\alpha}\rt|\leq D_2,
\ee
where $D_i,\,\,i=1,2$ are positive constants depending on $|\varphi|_{C^{n+3}(B'_{\sqrt{\delta}})},$ $\delta,$ and $\sigma.$

$\bf{Step \, 4}.$ Combining equations \eqref{fe.13}, \eqref{eq4.4}, and \eqref{eq4.15}--\eqref{eq4.18}
we get
\be\label{eq4.19}
\begin{aligned}
Q(u^*+\psi)&\leq\bQ^{ij}\psi_{ij}+\bQ^s\psi_s+Ct^{n+1}|\log t|\\
&\leq(1-\sigma^2)\frac{n-l}{n}|c_1|^{n-1}(t\psi_{tt}-n\psi_t)
+(1-\sigma^2)\La\mathbb{P}_{n-2}(\La)(t\psi_{tt}-n\psi_t)\\
&\qquad+(O(|\tL|^2)+tA_7)(|t\psi_{tt}|+|\psi_t|)+tD_1|\psi_{\alpha j}|+D_2|\psi_\alpha|+Ct^{n+1}|\log t|.\\
\end{aligned}
\ee
Since
\be\label{eq4.20}
\begin{aligned}
t\psi_{tt}-n\psi_t&=-An\theta^{n-1}\lt[(n+1)|y'|^2-(n+3)\theta^{\frac{1}{2}}|y'|^2\rt]-\frac{A}{2}(n+3/2)\theta^{n-\frac{1}{2}}t,
\end{aligned}
\ee
\be\label{eq4.21}
|\psi_{\alpha\beta}|\leq 4An(n+1)\theta^{n-1}|y'|^2,
\ee
\be\label{eq4.22}
|\psi_{\alpha t}|\leq 2An(n+1)\theta^{n-1}|y'|,
\ee
and
\be\label{eq4.23}
|\psi_\alpha|\leq 2A(n+1)\theta^n|y'|.
\ee
We get
\be\label{eq4.28}
\begin{aligned}
&(1-\sigma^2)\frac{n-l}{n}|c_1|^{n-1}(t\psi_{tt}-n\psi_t)\\
&\leq (1-\sigma^2)\frac{n-l}{n}|c_1|^{n-1}\lt[\frac{-An(n+1)}{2}\theta^{n-1}|y'|^2-\frac{A}{2}(n+3/2)\theta^{n-\frac{1}{2}}t\rt],
\end{aligned}
\ee
\be\label{eq4.24}
|(1-\sigma^2)\La\mathbb{P}_{n-2}(\La)(t\psi_{tt}-n\psi_t)|=O(A\theta^nt),
\ee
\be\label{eq4.25}
(O(|\tL|^2)+tA_7)(|t\psi_{tt}|+|\psi_t|)=O(A\theta^n t),
\ee
\be\label{eq4.26}
tD_1|\psi_{\alpha j}|=O(A\theta^{n-1}|y'|t),
\ee
and
\be\label{eq4.27}
D_2|\psi_\alpha|=O(A\theta^n|y'|).
\ee
Note that
\[\theta^{n-1}|y'|t\leq\e_1\theta^{n-1}|y'|^2+\frac{\delta^{1/2}}{\e_1}\theta^{n-\frac{1}{2}}t,\]
and
\[t^{n+1}|\log t|<t^{n+1/2}\leq\theta^{n-1/2}t.\]

$\bf{Step \, 5}.$ We conclude that there exists an $A=C_0=C_0(|\varphi|_{C^{n+3}(B'_{\sqrt{\delta}})}, \partial\Omega, \sigma)$ such that $Q(u*+\psi)<0$
in $G_{\sqrt{\delta},\delta}.$ By maximum principle we have
\[u\leq u^*+\psi.\]
Similarly, we can show that
\[u\geq u^*-\psi.\]
Thus, we obtain the desired result.
\end{proof}

 \bigskip
 \section{The tangential regularity}
 \label{tr}
In this section, we will study the regularity of $u$ along tangential directions.
Let $u\in C^1(\bar{G}_{1, 1})\cap C^\infty(G_{1,1})$ be a solution of equations
\eqref{fe.2} and \eqref{fe.2'}. By Theorem 1.2 in \cite{GS11}, we know that
\be\label{eq5.0}
|Du|+t|D^2u|\leq C.
\ee
Set
\be\label{eq5.1}
v=u-\varphi-c_1(y')t,
\ee
We can see that if $\varphi\in C^{l+3, \alpha}(B'_1),$ $l\geq 0,$ then $v\in C^1(\bar{G}_{1,1})\cap C^{l+2, \alpha}(G_{1,1}).$
We want to obtain a better regularity of $v$ of the tangential directions.
We will first estimate the derivatives of $v$ near the boundary.
\begin{lemma}
\label{lemtr.1}
Assume $\varphi\in C^{3,\alpha}(B'_1),$  and $u\in C^1(\bar{G}_{1, 1})\cap C^\infty(G_{1,1})$ be a solution of equations
\eqref{fe.2} and \eqref{fe.2'}. Let $v$ be defined as \eqref{eq5.1}. Then, for any $r\in (0, r_0)$ we have
\be\label{eq5.2}
\frac{|Dv|}{t}+|D^2 v|\leq C\,\,\mbox{in $G_{r/2, r/2},$}
\ee
where $C=C(|\varphi|_{C^{3, \alpha}(G_{1,1})}, |u|_{L^\infty}(G_{1,1})).$
\end{lemma}
\begin{proof}
We denote
\[m_{ij}^v=\frac{t}{w}\ga^{ik}u_{kl}\ga^{lj}-\frac{u_t}{w}\delta_{ij}=\frac{M_{ij}}{w},\]
\[\hF(m_{ij}^v)=\lt(\frac{H_n(m_{ij}^v)}{H_l(m_{ij}^v)}\rt)^\frac{1}{n-l}=\sigma,\]
and
\[\hG(D^2u, Du, t)=\frac{1}{t}\hF(m_{ij}^v)=\frac{\sigma}{t}.\]
By the formal calculations in Section 3 we have
\be\label{eq5.3}
\hG(D^2(u_1+v), D(u_1+v), t)-\hG(D^2u_1, Du_1, t)=A_7(y', t),
\ee
where $u_1=\varphi+c_1t$ and $A_7$ is bounded. We can also assume that $\hG(D^2u_1, Du_1, t)>0$ in $G_{r, r}$ for $r\leq r_0.$
Thus we get
\be\label{eq5.4}
\hG^{ij}v_{ij}+\hG^rv_r=A_7(y' t),
\ee
where
\[\hG^{ij}=\int_0^1\frac{\partial}{\partial u_{ij}}\hG(D^2(u_1+sv), D(u_1+sv), t)ds\]
and
\[\hG^r=\int_0^1\frac{\partial}{\partial u_r}\hG(D^2(u_1+sv), D(u_1+sv), t)ds.\]

By direct calculations we obtain
\be\label{eq5.5}
\frac{\partial\hG}{\partial u_{ij}}=\ga^{ik}\frac{\partial\hF}{\partial m^v_{kl}}\ga^{lj}=\ga^{ik}F^{kl}\ga^{lj},
\ee
and
\be\label{eq5.6}
\begin{aligned}
\frac{\partial\hG}{\partial u_r}&=\frac{1}{t}\hF^{ij}\frac{\partial m^v_{ij}}{\partial u_r}\\
&=\frac{1}{t}\hF^{ij}\lt(-\frac{M_{ij}}{w^2}\frac{\partial w}{\partial u_r}+\frac{1}{w}\frac{\partial M_{ij}}{\partial u_r}\rt)\\
&=-\frac{\sigma}{tw}\frac{u_r}{w}+\frac{1}{w}\hF^{ij}\lt[\frac{-2m_{ik}}{1+w}\lt(u_k\ga^{rj}+\frac{u_j\ga^{kr}}{w}\rt)-\delta_{rt}\delta_{ij}\rt],
\end{aligned}
\ee
where we used $\hF^{ij}m^v_{ij}=\sigma.$
By equation \eqref{eq5.0} and our assumption that $\hG(D^2u_1, Du_1, t)>0,$ we have
$\hG^{ij}$ is uniformly elliptic and $t|\hG^r|<C.$ We denote $P^r=t\hG^r,$ then $v$
satisfies
\be\label{eq5.7}
\hG^{ij}v_{ij}+\frac{P^rv_r}{t}=A_7(y', t).
\ee

Following the argument in \cite{HanJiang}, we take any $(y_0', t_0)\in G_{r/2, r/2}$ and set $\lambda=t_0/2.$
Let $e_n=(0', 1),$ we now consider the transform $T:B_1(e_n)\rightarrow G_{r, r}$ by
\[y'=y_0'+\lambda z',\,\,t=\lambda(s+1).\]
Set $u^\lambda(z', s)=\lambda^{-2}u(y', t)$ and $v^\lambda(z', s)=\lambda^{-2}v(y', t).$ Then we have
\[D_{y}u=\lambda D_{z}u^\lambda,\,\,\mbox{and $D^2_{y}u=D^2_{z}u^{\lambda},$}\]
and
\[D_{y}v=\lambda D_{z}v^\lambda,\,\,\mbox{and $D^2_{y}v=D^2_{z}v^{\lambda}.$}\]
Moreover, $u^\lambda$ satisfies
\[\hG(\lambda D^2_{z}u^\lambda, \lambda^2D_z u^\lambda, \lambda^2(s+1))=\frac{\sigma}{s+1}.\]
Applying the Evans-Kylov interior estimates we obtain $\lambda u^\lambda\in C^{2, \alpha}(B_{1/2}(e_n)),$ which in turn implies that
$\hG^{ij}, P^r$ in \eqref{eq5.7} are $C^\alpha.$

Now, let's consider $v^\lambda.$ Since $v^\lambda$  satisfies
\be\label{eq5.8}
\hG^{ij}v^\lambda_{ij}+\frac{P^rv^\lambda_r}{1+s}=A_7(y_0'+\lambda z', \lambda(s+1) ),
\ee
and by our discussions above, we have $\hG^{ij}, P^r,$ $A_7$ are $C^\alpha.$ By the standard Schauder estimate,
we obtain $v^\lambda\in C^{2, \alpha}(B_{1/2}(e_n)).$
We evaluate $v^\lambda$ at $(0, 1)$ and get
\[\frac{|Dv(y_0', t_0)|}{t_0}+|D^2v(y_0', t_0)|\leq C,\]
this proves \eqref{eq5.2} since $(y_0', t_0)$ is an arbitrary point in $G_{r/2, r/2}.$
\end{proof}

Next, similar to the proof of Theorem \ref{th4.1}, we consider
\[\bQ(D^2(u_1+v), D(u_1+v), t)-\bQ(D^2u_1, Du_1, t)=B_1(y',t).\]
Here and in the following, we denote $B_i's,$ $i\in\mathbb{Z}^+$ as some functions which are smooth in their arguments,
and are essentially polynomials of $tD^2 v, Dv, t$, with coefficients depending on $y^\prime$. Moreover, all monomials in $B_i$ have at least one $tD^2 v, Dv$ or $t$ factor, i.e., $B_i=O(t).$
Therefore, we have
\be\label{eq5.9}
\begin{aligned}
&\quad\int_0^1\frac{\partial\bQ}{\partial u_{ij}}ds \cdot v_{ij}+\int_0^1\frac{\partial\bQ}{\partial u_i}(u_1+sv)ds \cdot v_i\\
&=\bQ_{ij}v_{ij}+\bQ_i v_i=B_1(y', t).\\
\end{aligned}
\ee
It's easy to check
\be\label{eq5.10}
\bQ_{tt}=t(1-\sigma^2)\frac{n-l}{n}|c_1|^{n-1}+tB_2(tD^2v, Dv, y', t)
\ee
and
\be\label{eq5.11}
\bQ_t=-(n-l)(1-\sigma^2)|c_1|^{n-1}+B_3(tD^2v, Dv, y', t).
\ee
Thus, $v$ satisfies the following equation
\be\label{eq5.12}
t^2 \tQ_{ij}v_{ij}+t \tQ_tv_t- tB_1(y^\prime, t)+ tB_4(tD^2v, Dv, y', t) =0,
\ee
where $\tQ_{tt}=1$ and $\tQ_t=-n.$ Monomials in $B_4$ have at least two $tD^2 v, Dv$ or $t$ factors, i.e. $B_4=O(t^2).$

Applying Theorem 4.3 in \cite{HanJiang} we conclude:
\begin{theorem}
\label{th5.1}
Assume $\varphi\in C^{l+3, \alpha}(B_1'), l\geq 0.$ Let $u\in C^1(\bar{G}_{1,1})\cap C^\infty(G_{1,1})$ be a solution of
equations \eqref{fe.2} and \eqref{fe.2'} in $G_{1,1},$ and $v=u-u_1\in C^1(\bar{G}_{1,1})\cap C^{l+2, \alpha}(G_{1,1}).$
Then, for any  $r\in (0, r_0)$ such that for $\tau=0, 1, \cdots, l,$
\be\label{eq5.14}
\frac{D^{\tau}_{y'}v}{t^2},\,\,\frac{DD^{\tau}_{y'}v}{t},\,\,D^2D^{\tau}_{y'}v\in C^\alpha(\bar{G}_{r, r}),
\ee
and
\be\label{eq5.15}
\frac{D^{\tau}_{y'}v}{t},\,\,D^{\tau}_{y'}Dv,\,\,\frac{D^{\tau}_{y'}(v^2)}{t^3},\,\,
\frac{D^{\tau}_{y'}(vv_t)}{t^2},\,\,\frac{D^{\tau}_{y'}(v_t^2)}{t}\in C^{1,\alpha}(\bar{G}_{r, r}),
\ee
where $0<r_0<1$ depends on $\partial\Omega$ and $\sigma.$
\end{theorem}

A remark is that under the same assumption as Theorem \ref{th5.1}, we can show $D^2 D^{l+1}_{y^\prime} u \in C^{\alpha}(\bar{G}_{r,r})$, where $l$ can start from $-1$. To this end, we need to freeze the coefficient $c_1(y^\prime)$ in \eqref{eq5.1} by defining,
 \begin{align*}
 \bar{v}=u-\bar{u}_1=  u-\varphi-c_1(0^\prime)t.
 \end{align*}
Then it follows from Theorem 4.3 in \cite{HanJiang}.

\bigskip
\section{Regularity along the normal direction and convergence}
\begin{subsection}{Regularity along the normal direction}
\label{subnr}
In this subsection, we will discuss the regularity along the normal direction.

Assuming the same assumption as in Theorem \ref{th5.1}, we have shown that $v=u-u_1$ satisfies the equation \eqref{eq5.12}, and has estimates
\eqref{eq5.14},\eqref{eq5.15}.

Denote $v_1=v^\prime$.
We differentiate \eqref{eq5.12} with respect to $t$, divide it by $1+ \frac{\partial B_4}{\partial (tv_{tt})}$, and rewrite the equation as an ODE of $v_1$,
\begin{align}\label{eq-ODE}
&\qquad t^2 v_{1}^{\prime\prime}- (n-2) tv_{1}^\prime -nv_1 = \bar{F},
\end{align}
where $\overline{F}$ has the form,
\begin{align*}
\bar{F}:= t^2 F(y, v_1^\prime, v_1, D^2_{y'}v_1, D_{y'} v_1/t, D^2_{y'}v/t, D_{y'} v/t) + t A(y^\prime, t),
\end{align*}
where $A=B_1/t+B_1^\prime$ is a smooth function.
In fact, after differentiating \eqref{eq5.12} by $t$, $v_{ttt}$ has coefficient
\begin{align*}
t^2 + t^2 \frac{\partial B_4}{\partial (tv_{tt})} &= t^2 (1+  \frac{\partial B_4}{\partial (tv_{tt})})
\end{align*}
where $\frac{\partial B_4}{\partial (tv_{tt})}=O(t)$ is a degree one polynomial in $tD^2v, Dv, t$.
Then we divide the equation by $1+ \frac{\partial B_4}{\partial (tv_{tt})}$ to derive \eqref{eq-ODE}.

Inductively, we denote $v_k= D^k_t v$ for $k\geq 2$, keep differentiating \eqref{eq-ODE} with respect to $t$, and derive an ODE of $v_k$,
\begin{align*}
&\qquad
t^2 v_k^{\prime\prime}-(n-2k) t v_k^\prime -(n+1-k)k v_k= \bar{F}^{(k-1)}.
\end{align*}

Applying Theorem 5.3 in \cite{HanJiang} we show
\begin{theorem}
\label{th6.1}
Assume $\varphi\in C^{p, \alpha}(B_1'),$ $p\geq k\geq n+1.$ Let $u\in C^1(\bar{G}_{1,1})\cap C^\infty(G_{1,1})$ be a solution of
equations \eqref{fe.2} and \eqref{fe.2'} in $G_{1,1},$ and $v=u-u_1\in C^1(\bar{G}_{1,1})\cap C^{p-1}(G_{1,1}).$
Then, there exists a positive constant $0<r<1,$ such that for any $(y', t)\in G_{r, r},$
\begin{align}\begin{split}\label{eq6.2}
v(y', t) &=\sum_{i=2}^nc_i(y')t^i+\sum_{i=n+1}^k \sum_{j=0}^{N_i} c_{i,j}(y')t^i (\log t)^j\\&\qquad+\int_0^t\cdots\int_0^{s_{k-1}}
w_k(y', s_k)ds_kds_{k-1}\cdots ds_1,
\end{split}
\end{align}
where $c_i,$ $c_{i, 0}$ are $C^{p-i, \alpha}(B'_r)$ for $i\leq n+1;$ and $c_i$,$c_{i, j}$ are $C^{p-i, \e}(B'_r)$ for any $\e\in(0,\alpha).$
In addition, $w_k$ is a function in $G_{r, r}$ such that for any $\tau=0, 1, \cdots, p-k$ and any $\e\in(0, \alpha),$
\be\label{eq6.3}
D^{\tau}_{y'}w_k\in C^\e(\bar{G}_{r, r}),
\ee
and
\be\label{eq6.4}
|D^{\tau}_{y'}w_k|\leq t^\alpha\,\,\mbox{in $G_{r, r}$}.
\ee
\end{theorem}
\end{subsection}
\begin{subsection}{Convergence}
\label{subcon}
In the following, we assume $\varphi$ is analytic in $B_1^\prime$. By \cite{HanJiang2}, the series derived from the expansion
\begin{align*}
c_2(y')t^2+\cdots+  c_{n}(y')t^{n}+  \sum_{i=n+1}^k \sum_{j=0}^{N_i} c_{i,j}(y')t^i (\log t)^j + \cdots,
\end{align*}
converges to $v$ in $G_{r_1, \delta_1}$. Furthermore, $v$ is analytic in
\begin{align*}
y^\prime, t, t\log t
\end{align*}
for $(y^\prime, t)$ in $G_{r_1, \delta_1}$. To this end, first we apply Theorem 2.1 in \cite{HanJiang2} to show for any integer $l\geq 0$,
\begin{align*}
|D_{y^\prime}^l v(y^\prime, t) |\leq B_0 B^{(l-1)^+}(l-1)!t^2 (1-|y^\prime|)^{-(l-1)^+},\\
|D D_{y^\prime}^l v(y^\prime, t)| \leq B_0 B^{(l-1)^+}(l-1)!t (1-|y^\prime|)^{-(l-1)^+},\\
|D^2 D_{y^\prime}^l v(y^\prime, t)| \leq B_0 B^{(l-1)^+}(l-1)! (1-|y^\prime|)^{-(l-1)^+},
\end{align*}
in $G_{1,1}$, where $B_0, B$ are independent of $l$. Then we derive Theorem \ref{thm-MainConv}  by applying Theorem 4.1 in \cite{HanJiang2}.

\end{subsection}


\begin{thebibliography}{99}



\bibitem{Guan&Spruck&Szapiel} Bo Guan, Joel Spruck, Marek Szapiel,
\emph{Hypersurfaces of constant curvature in
hyperbolic space I},
J Geom Anal (2009) 19: 772-795.

\bibitem{GS11} Guan, Bo; Spruck, Joel
{\em Convex hypersurfaces of constant curvature in hyperbolic space.} Surveys in geometric analysis and relativity, 241-–257,
Adv. Lect. Math. (ALM), 20, Int. Press, Somerville, MA, 2011.

\bibitem{Guan&Spruck} Bo Guan, Joel Spruck,
\emph{Hypersurfaces of constant curvature in
hyperbolic space II},
J. Eur. Math. Soc. 12, 797-817.

\bibitem{GS00}Guan, Bo; Spruck, Joel
{\em Hypersurfaces of constant mean curvature in hyperbolic space with prescribed asymptotic boundary at infinity.}
Amer. J. Math. 122 (2000), no. 5, 1039-1060.

\bibitem{GSX} Guan, Bo; Spruck, Joel; and Xiao, Ling,
{\em Interior curvature estimates and the asymptotic plateau problem in hyperbolic space,}
J. Differential Geom. 96 (2014), no. 2, 201-222.


\bibitem{HanJiang} Q. Han, X. Jiang, {\it Boundary expansions for minimal graphs in the hyperbolic space}, 	arXiv:1412.7608.

\bibitem{HanJiang2} Q. Han, X. Jiang, {\it The convergence of boundary expansions and the analyticity of minimal surfaces in the hyperbolic space}, arXiv:1801.08348.



\bibitem{HanWang} Han, Wang, {\it Boundary Expansions for Constant Mean Curvature Surfaces in the Hyperbolic Space}, arXiv:1608.07803.


\bibitem{Lin1989Invent} F.-H. Lin,
{\it On the Dirichlet problem for minimal graphs in hyperbolic space},
Invent. Math., 96(1989), 593-612.
\end{thebibliography}
\end{document}